\documentclass[12pt]{amsart}

\usepackage{fullpage}
\usepackage{latexsym,amssymb,amsfonts,amsmath,mathrsfs,float}
\usepackage[font=small,labelfont=bf, width=.618\textwidth]{caption} 
\usepackage{xcolor}
\usepackage{tikz, graphics, enumerate}
\usepackage{hyperref}

  \DeclareMathOperator{\var}{Var}
  \renewcommand{\Pr}{\mbox{\rm Pr}}	
  \newcommand{\Exp}{{\mathbb{E}}}

  \newcommand{\R}{\mathbb{R}} 
  \newcommand{\N}{\mathbb{N}} 
  \newcommand{\Z}{\mathbb{Z}} 
  \newcommand{\F}{\mathbb{F}} 
  \newcommand{\pmset}[1]{\{-1,1\}^{#1}} 
  
  \DeclareMathOperator{\Cball}{\mathcal C}
  \DeclareMathOperator{\cay}{Cay} 
  \DeclareMathOperator{\sol}{Sol} 

  \newcommand{\st}{:\,} 

  \newcommand{\eps}{\varepsilon}

  \newcommand{\ceil}[1]{\lceil{#1}\rceil}

  \DeclareMathOperator{\diam}{diam}

  \newcommand{\pnorm}{{\ell_p,\dots,\ell_p}}
  \newcommand{\tnorm}{{\ell_t,\dots,\ell_t}}

  \newcommand{\beq}{\begin{equation}}
  \newcommand{\eeq}{\end{equation}}
  \newcommand{\beqn}{\begin{equation*}}
  \newcommand{\eeqn}{\end{equation*}}
  \newcommand{\beqr}{\begin{eqnarray}}
  \newcommand{\eeqr}{\end{eqnarray}}
  \newcommand{\beqrn}{\begin{eqnarray*}}
  \newcommand{\eeqrn}{\end{eqnarray*}}
  \newcommand{\bmline}{\begin{multline}}
  \newcommand{\emline}{\end{multline}}
  \newcommand{\bmlinen}{\begin{multline*}}
  \newcommand{\emlinen}{\end{multline*}}

  \theoremstyle{plain}
  \newtheorem{theorem}{Theorem}[section]
  \newtheorem{lemma}[theorem]{Lemma}
  \newtheorem{proposition}[theorem]{Proposition}
  
  \newtheorem{corollary}[theorem]{Corollary}
  
  \theoremstyle{definition}
  \newtheorem{definition}[theorem]{Definition}

  \theoremstyle{remark}
  \newtheorem{remark}[theorem]{Remark}
  \renewenvironment{proof}[1][]{
    	\begin{trivlist}
     	\item[\hspace{\labelsep}{\em\noindent Proof#1:\/}]}
     	{{\hfill$\Box$}
    	\end{trivlist}
  }
  
  \makeatletter
  \newtheorem*{rep@theorem}{\rep@title}
  \newcommand{\newreptheorem}[2]{%
  \newenvironment{rep#1}[1]{%
  \def\rep@title{#2 \ref{##1}}%
  \begin{rep@theorem}}%
  {\end{rep@theorem}}}
  \makeatother

  \newreptheorem{lemma}{Lemma}
  \newreptheorem{theorem}{Theorem}
  \newreptheorem{corollary}{Corollary}
  \newreptheorem{proposition}{Proposition}
  \newreptheorem{conjecture}{Conjecture}

\begin{document}

\title[Arithmetic expanders and deviation bounds for random tensors]{Arithmetic expanders and deviation bounds for sums of random tensors}

\author[Jop Bri\"{e}t]{Jop Bri\"et}
\address{CWI, Science Park 123, 1098 XG Amsterdam, The Netherlands}
\email{j.briet@cwi.nl}
\thanks{J.~B.\ was supported by a VENI grant from the Netherlands Organisation for Scientific Research (NWO)}

\author{Shravas Rao}
\address{Courant Institute, New York University, 251 Mercer Street, New York NY 10012, USA}
\email{rao@cims.nyu.edu}
\thanks{This material is based upon work supported by the National Science Foundation Graduate Research Fellowship Program under Grant No. DGE-1342536.}

\date{}

\maketitle

\begin{abstract}
We prove hypergraph variants of the celebrated Alon--Roichman theorem on spectral expansion of sparse random Cayley graphs.
One of these variants implies that for every prime $p\geq 3$ and any $\eps > 0$, there exists  a set of directions $D\subseteq \F_p^n$ of size $O_{p,\eps}(p^{(1-1/p +o(1))n})$ such that
for every set~$A\subseteq \F_p^n$ of density~$\alpha$,
the fraction of lines in~$A$ with direction in~$D$
is within $\eps\alpha$
of the fraction of all lines in~$A$.
Our proof uses new deviation bounds for sums of independent random multi-linear forms taking values in a generalization of the Birkhoff polytope.
The proof of our deviation bound is based on Dudley's integral inequality and a probabilistic construction of $\eps$-nets.
Using the polynomial method we prove that a Cayley hypergraph with edges generated by a set~$D$ as above requires~$|D| \geq \Omega_p(n^{p-1})$ for (our notion of) spectral expansion for hypergraphs.

%
%
\end{abstract}

\section{Introduction}

In the following all graphs are undirected and may have loops and parallel edges.
For an $n$-vertex graph~$G = (V,E)$ and $u,v\in V$ denote by $e_G(u,v)$ the number of edges connecting~$u$ and~$v$.
If~$G$ is~$k$-regular then its normalized adjacency matrix~$A_G \in \R^{V\times V}$  is given by~$A_G(u,v) = e_G(u,v)/k$.
Let~$1 = \lambda_1(G) \geq \lambda_2(G) \geq \cdots \geq \lambda_{n}(G)\geq -1$ be the eigenvalues of~$A_G$ arranged in decreasing order
and
denote $\lambda(G) = \max_{i\in \{2,\dots,n\}}|\lambda_i(G)|$.

\subsection{Spectral expanders.}
Spectral expanders are infinite families of graphs $\{G_i\}_{i\in\N}$ of size increasing with~$i$ such that the spectral gap $1 - \lambda(G_i)$ is at least some $\delta > 0$ that is independent of~$i\in\N$.
A single graph is said to be an expander if it is tacitly understood to belong to such a family.
Spectral expansion, the property of having large spectral gap, occurs in random graphs have with high probability.
Seminal work on quasirandomness of Thomason~\cite{Thomason:1987, Thomason:1987b}, and Chung, Graham, and Wilson~\cite{Chung:1989} showed that for dense graphs, this property is equivalent to a number of other likely features of random graphs.
One of these is \emph{expansion}, a measure of connectedness showing that no large set of vertices can be disconnected from its complement by cutting only a few edges.
Another is \emph{discrepancy}, which refers the property that the edge density of any sufficiently large induced subgraph is close to the overall edge density.

A long line of research extending the results of~\cite{Chung:1989} to dense hypergraphs was initiated by Chung and Graham~\cite{Chung:1990}, culminating in recent work of Lenz and Mubayi~\cite{Lenz:2015, Lenz:2015_eig} (which we refer to for a more detailed account).
Partially motivated by an application in Theoretical Computer Science concerning special types of error-correcting codes (locally decodable codes)~\cite{BrietDG:2017}, we study the extent to which some known results on \emph{sparse} expanders generalize to hypergraphs.
Along the way we establish a new deviation inequality for sums of independent random multi-linear forms (Theorem~\ref{thm:tensor-hoeffding}) that we hope will find applications elsewhere.

\subsection{Cayley graphs and the Alon--Roichman Theorem.}\label{sec:AR}
Most known examples of sparse  expanders are Cayley graphs, which are defined as follows.
For a finite group~$\Gamma$ and an element~$g\in \Gamma$, the Cayley graph $\cay(\Gamma, \{g\})$ is the 2-regular graph with vertex set~$\Gamma$ and edge set $\{\{u, gu\} \st u\in \Gamma\}$, where in case $g^2 = 1$, all edges are doubled.
For a multiset\footnote{We use curly brackets to delimit \emph{multisets}: unordered lists that may contain repeated elements.} $S = \{g_1,\dots,g_k\}\subseteq \Gamma$, the  Cayley graph $\cay(\Gamma,S)$  is the $2k$-regular graph formed by the  union of the graphs $\cay(\Gamma, \{g_1\}),\dots, \cay(\Gamma, \{g_k\})$.

The group over which Cayley graphs are defined strongly influences the minimal degree required for spectral expansion.
The famous examples of constant-degree  expanders of Margulis~\cite{
Margulis:1973, Margulis:1988} and Lubotzky, Phillips, and Sarnak~\cite{Lubotzky:1988} are Cayley graphs which, crucially, are defined over non-Abelian groups.
It is easy to see that a Cayley graph over the Abelian group~$\F_2^n$, for example, requires degree at least~$n$ to be an expander~\cite{Alon:1994a}.

\begin{proposition}\label{prop:F2n}
Let~$G = \cay(\F_2^n, S)$ be such that $|S| < n$. Then,~$\lambda(G) = 1$.
\end{proposition}

\begin{proof}
Let $\Gamma = \F_2^n$. 
Let~$T\subseteq \Gamma$ be an $(n-1)$-dimensional subspace containing~$S$ and let $\overline T = \Gamma\smallsetminus T$.
Since $u,v\in\Gamma$ are connected if and only if $u-v\in S$ and every pair $u\in T, v \in \overline T$ satisfies $v - u\in \overline T$, the sets~$T$ and~$\overline T$ are disconnected.
It follows that $1_\Gamma - 21_T$ is an eigenvector of $A(G)$ 
and has eigenvalue~$-1$.
Hence, $\lambda(G) = 1$.
\end{proof}

Similarly, because expanders must be connected, it follows that spectral expansion requires degree $\Omega(\log n)$ in any Cayley graph over any $n$-element Abelian group~\cite[Proposition~11.5]{Hoory:2006}.
A celebrated result of Alon and Roichman~\cite{Alon:1994a}, however, shows that Abelian groups are extreme in this sense.

\begin{theorem}[Alon--Roichman Theorem]\label{thm:ar-spec}
For any~$\eps \in (0,1)$ there exists a $c(\eps) \in (0,\infty)$ such that the following holds.
Let~$\Gamma$ be a finite group of cardinality~$n$.
Let~$k \geq c(\eps)\log n$ be an integer and let~$g_1,\dots, g_k$ be independent uniformly distributed elements from~$\Gamma$.
Then, with probability at least~$1/2$, the Cayley graph~$G = \cay(\Gamma, \{g_1,\dots,g_k\})$ satisfies
$
\lambda(G) \leq \eps.
$
\end{theorem}

Our main results are 
hypergraph versions of Proposition~\ref{prop:F2n} and Theorem~\ref{thm:ar-spec}.

\subsection{Hypergraphs}
\label{sec:hypergraphs}
A $t$-uniform hypergraph $H = (V, E)$ with vertex set~$V$ has as edge set~$E$ a family of unordered $t$-element multisets with possible parallel edges.
For $u_1,\dots,u_t\in V$ let $e_H(u_1,\dots,u_t)$ denote the number of  edges equal to $\{u_1,\dots,u_t\}$.
The  \emph{adjacency form} of~$H$ is the  $t$-linear form ${\overline{A}_H:\R^V\times\cdots\times\R^V\to\R}$ defined by $\overline{A}_H(1_{\{u_1\}},\dots,1_{\{u_t\}}) = e_H(u_1,\dots,u_t)$.
The degree of a vertex~$v\in V$ is defined by $\overline{A}_H(1_{\{v\}}, 1_V,\dots,1_V)$ and $H$ is $k$-regular if every vertex has degree exactly $k$, in which case its normalized adjacency form is $A_H = \overline{A}_H/k$.
Of particular importance here are hypergraphs whose edge set is given by a multiset of the form $\{\pi_1(v),\dots, \pi_t(v)\}$, $v\in V$, where $\pi_1,\dots,\pi_t$ are permutations on~$V$.
In this case we set
\beq\label{eq:permcount}
e_H(u_1,\dots,u_t)
=
\sum_{\sigma}
\sum_{v\in V}
1_{\{u_1\}}\big(\pi_{\sigma(1)}(v)\big)
\cdots
1_{\{u_t\}}\big(\pi_{\sigma(t)}(v)\big),
\eeq
where~$\sigma$ runs over all permutations of~$[t] = \{1,\dots, t\}$,
giving a $(t!)$-regular hypergraph.

\subsection{Hypergraph spectral expansion.}
\label{sec:hypexp}
To define spectral expansion for hypergraphs we build on the following characterisation of~$\lambda(G)$.
Recall that the Schatten-$\infty$ norm (or spectral norm) of a  matrix~$A$ is given by
$
\|A\|_{S_\infty} = \sup_{x,y\in \R^n\smallsetminus \{{\bf 0}\}}|x^{\mathsf T}Ay|/\|x\|_{\ell_2}\|y\|_{\ell_2}
$.
If~$A$ is symmetric, then this norm is precisely the maximum absolute value of the eigenvalues of~$A$.
Since for an $n$-vertex graph~$G$, the eigenvector associated with the first eigenvalue~$\lambda_1(G) = 1$ is the normalized all-ones vector~${\bf 1}/\sqrt{n}$, we have $\lambda(G)=\|A_G - J/n\|_{S_\infty}$, where $J = {\bf 1}{\bf 1}^{\mathsf T}$ is the all-ones matrix.
Our definition of spectral expansion for hypergraphs is based on the following norm on multilinear forms.
For a $t$-linear form~$A$ on $\R^n$ and $p\in[1,\infty]$ define
\beqn
\|A\|_\pnorm
=
\sup\Big\{\frac{A(x[1],\dots,x[t])}{\|x[1]\|_{\ell_{p}}\cdots \|x[t]\|_{\ell_{p}}}\st x[1],\dots,x[t] \in \R^n\smallsetminus \{{\bf 0}\}\Big\}.
\eeqn
The notion of spectral expansion we shall use is relative to a fixed regular $t$-uniform hypergraph~$K$.
In particular, for a regular $t$-uniform hypergraph~$H$, we define
\beq\label{eq:lambdaKH}
\lambda_K(H)
=
\|A_H - A_K\|_\tnorm
.
\eeq
For graphs, this parameter coincides with~$\lambda(G)$ if~$K$ is the complete graph with all loops.

\subsection{Cayley hypergraphs.}
\label{sec:hypcal}
A Cayley hypergraph over a finite group~$\Gamma$ is a disjoint union of particular permutation hypergraphs as mentioned in Section~\ref{sec:hypergraphs}.
Let $q \in (\Z\smallsetminus\{0\})^t$ be an integer vector such that no element of~$\Gamma$ has order~$q_j$ for every $j\in[t]$. 
This ensures that for every $g\in \Gamma$, the maps $u\mapsto u^{q_j}g$ are permutations.
For ${\bf g} = (g[1],\dots,g[t])\in\Gamma^t$, we define $\cay^{(t)}(\Gamma, q, {\bf g})$ to be the hypergraph as in Section~\ref{sec:hypergraphs} based on the permutations $\pi_j(u) = u^{q_j}g[j]$.
For a multiset $S = \{{\bf g}_1,\dots,{\bf g}_k\}\subseteq \Gamma^{t}$, we let $\cay^{(t)}(\Gamma, q, S)$ be the $(t!)k$-regular hypergraph given by the union of $\cay^{(t)}(\Gamma, q, \{{\bf g}_i\})$ for $i\in[k]$.\\

To connect the above definitions, consider a Cayley hypergraph $K =\cay^{(t)}(\Gamma, q, S)$.
For a subset $S'\subseteq S$, let $H = \cay^{(t)}(\Gamma, q, S')$ be a sub-hypergraph of~$K$ and let $\eps = \lambda_K(H)$.
Then, 
for every set $T\subseteq V$ of density $\tau = |T|/|V|$, 
we have $|(A_H - A_K)(1_T,\dots,1_T)| \leq \eps|T|$. 
Dividing by~$|V|$ shows that the fraction of edges that~$T$ induces in~$H$ is within~$\eps\tau$ of the fraction of edges it induces in~$K$.

\subsection{Translation invariant equations}
\label{sec:transeqs}

To motivate the above definitions we focus on a special class of Cayley hypergraphs that arises from systems of translation invariant equations.
Such a system can be given in terms of a matrix $C\in \Z^{s\times t}$ and a vector $q\in (\Z\smallsetminus\{0\})^t$ such that $Cq = 0$.
For an Abelian group~$\Gamma$ without elements of order $q_j$ for every $j\in[t]$, we then consider the set of solutions in $\Gamma^t$ to the linear equations defined by~$C$, 
\beqn\label{eq:trsystem}
\sol(C) = \big\{{\bf h} = (h[1],\dots,h[t])\in \Gamma^t \st C{\bf h} = 0\big\}.
\eeqn

There is a large body of literature on the problem of bounding the maximum size of a set $A\subseteq \Gamma$ such that $\sol(C)\cap A^t$ contains only trivial solutions.
Well-studied examples involving a single equation (where $s = 1$) include 
Sidon sets~\cite{OBryant:2004}, where $C = [1,1,-1,-1]$, and sets without 3-term arithmetic progressions (APs), where $C = [1,-2,1]$ (sometimes referred to as cap sets)~\cite{Obryant:2011, Sanders:2011, Ellenberg:2016}. 
Sets avoiding a general $t$-variate  translation invariant equation were studied in~\cite{Ruzsa:1993, Bloom:2012, Schoen:2014}.
Probably the most-studied examples involving more than one equation are $t$-term APs \cite{Szemeredi:1990, Green:2007, Tao:2007, Obryant:2011}, where
\beq\label{eq:CtAP}
C
=
\begin{bmatrix}
1 & -2 & 1 & 0 & \cdots & 0 &0 &0\\
0 & 1 & -2 & 1 & \cdots & 0 &0 &0\\
\vdots &\vdots&\vdots&\vdots& \ddots & \vdots &\vdots&\vdots\\
0&0&0&0&\cdots& 1 & -2 & 1
\end{bmatrix}
\in \Z^{(t-2)\times t}.
\eeq

Translation invariance refers to the fact that for every ${\bf h}$ in~$\sol(C)$  and every $u\in \Gamma$, the tuple $(q_1u+ h[1],\dots, q_tu+ h[t])$ belongs to~$\sol(C)$ as well.
As such, $\sol(C)$ is a union of cosets of the subgroup $\{(q_1u,\dots, q_tu)\st u\in \Gamma\} \subseteq\Gamma^t$.
If $S = \{{\bf g}_1,\dots,{\bf g}_k\}$ is a set of representatives of these cosets, then the edge set of the hypergraph $K = \cay^{(t)}(\Gamma, q, S)$  is furnished precisely by the (unordered) tuples in~$\sol(S)$,
which leads to the following definition.

\begin{definition}[Arithmetic expander]\label{def:arithexp}
Let $K$ be the Cayley hypergraph as above.
A multiset $S'\subseteq S$ is a \emph{$(C,q,\Gamma,\eps)$-arithmetic expander} if 
$$\lambda_{K}\big(\cay^{(t)}(\Gamma, q, S')\big) \leq \eps.$$
\end{definition}

The preceding discussion shows that an arithmetic expander has the property that for every set~$A\subseteq \Gamma$ of density~$\alpha$, the fraction of solutions in~$A$ among the cosets represented by~$S'$ is within $\eps \alpha$ of the fraction of all solutions in~$A$.
For APs, this means the following.
The matrix~$C$ as in~\eqref{eq:CtAP} satisfies $Cq = 0$ for $q = {\bf 1} = (1,\dots,1)$,
from which it follows that $\sol(C)$ consists of cosets represented by APs through zero, $S = \{\{0,v,2v,\dots,(t-1)v\}\st v\in \Gamma\}$, which correspond to the possible steps that an AP can take.
In this case, an arithmetic expander is thus characterized by a small set of steps~$D\subseteq \Gamma$ 
such that the fraction of APs in any set~$A$ taking steps from~$D$ gives an accurate estimate of the fraction of all APs in~$A$.
The AP matrix~$C$ also satisfies $Cq= 0$ for $q = (1,2,\dots,t)$, from which it follows that $\sol(C)$ consists of the cosets with representatives given by the points through which $(t+1)$-term APs travel, $S = \{\{u,\dots,u\}\st u\in\Gamma\}$.
In this case, an arithmetic expander thus estimates the fraction of all APs by the fraction of APs travelling through a small fixed set of points.

\section{Our results}
\label{sec:ourresults}

\subsection{Spectral expansion of Cayley hypergraphs.}
Our first result is an extension of Proposition~\ref{prop:F2n} concerning arithmetic expanders for $t$-APs where $t$ is a prime.


\begin{theorem}\label{thm:apexp}
For every prime~$p$ there exist $\eps(p),\delta(p)\in (0,\infty)$ such that the following holds.
Let $n \geq p^2$ be an integer, let $\Gamma = \F_p^n$ and
let $C$ be as in~\eqref{eq:CtAP} with $t = p$.
Then, for any $\eps < \eps(p)$, any $(C,{\bf 1},\Gamma,\eps)$-arithmetic expander has size at least $\delta(p)\, n^{p-1}$.
\end{theorem}


Our second result is a version of Theorem~\ref{thm:ar-spec}, showing for instance that in the AP case, for~$C$ as in~\eqref{eq:CtAP}, there exist $(C, q, \F_p^n, \eps)$-arithmetic expanders of size $c(t,\eps)p^{(1 - 1/t + o(1))n}$ for both options of~$q$, where $c(t,\eps)$ depends on $t$ and $\eps$ only.

\begin{theorem}\label{thm:hypar}
For every integer $t\geq 3$ and $\eps \in (0,1)$ there exists a $c(t,\eps)\in (0,\infty)$ such that the following holds.
Let $\Gamma$ be a finite group of cardinality~$n$, let $q\in(\Z\smallsetminus \{0\})^t$ be such that $\Gamma$ has no elements of order $q_j$ for every~$j\in[t]$,  let $S\subseteq \Gamma^{t}$ be a multiset and $K = \cay^{(t)}(\Gamma, q,S)$.
For $k = c(t,\eps)n^{1 - 1/t}(\log n)^{t + 1/2}$, let $S'\subseteq S$ be a multi-set of~$k$ independent uniformly distributed tuples from $S$.
Then, with probability at least~$1/2$,  $\lambda_K(\cay^{(t)}(\Gamma, q,S')) \leq \eps$.
\end{theorem}

\subsection{A deviation bound for sums of random tensors}

Our proof of Theorem~\ref{thm:hypar} follows similar lines as
a slick proof of Theorem~\ref{thm:ar-spec} due to Landau and Russel~\cite{Landau:2004}.
Their proof is based on a matrix-valued deviation inequality called the matrix-Chernoff bound.
One can also use the following matrix version of the Hoeffding bound, which follows from a non-commutative  Khintchine inequality of Tomczak-Jaegermann~\cite{Tomczak-Jaegermann:1974} (see Appendix~\ref{sec:matrix-Hoeffding})
and which is more in line with the tools we shall use below.

\begin{theorem}[Matrix Hoeffding bound]\label{thm:matrix-Hoeffding}
There exist absolute constants~$c,C\in (0,\infty)$ such that the following holds.
Let~$A_1,\dots, A_k\in\R^{n\times n}$ be independent random matrices such that~$\|A_i\|_{S_\infty} \leq 1$ for each~$i\in[k]$.
Then, for any~$\eps>0$, we have
\beqn
\Pr\Big[\Big\|\frac{1}{k}\sum_{i=1}^k \big(A_i - \Exp[A_i]\big)\Big\|_{S_\infty} > \eps\Big]
\leq
C\exp\Big(-\frac{ck\eps^2}{\log n}\Big).
\eeqn
\end{theorem}

\begin{proof}[ of Theorem~\ref{thm:ar-spec}]
For~$s\in \Gamma$ let~$A_s \in \R^{\Gamma\times\Gamma}$ denote the adjacency matrix of the Cayley graph~$\cay(\Gamma,\{s\})$.
Since~$A_s$ is the average of two permutation matrices, $\|A_s\|_{S_\infty} \leq 1$.
Observe that if~$s\in \Gamma$ is a uniformly distributed element, then~$\Exp[A_s] = J/n$.
Moreover, since $A_G=(A_{g_1} + \cdots + A_{g_k})/k$,
the result now follows from Theorem~\ref{thm:matrix-Hoeffding}.
\end{proof}

The proof of Theorem~\ref{thm:hypar} is similarly based on a new deviation bound for multi-linear forms that belong to a generalization of the Birkhoff polytope (of doubly-stochastic matrices).
To define this polytope, we first consider the following generalization of a doubly-stochastic matrix.
Let $e_1,\dots, e_n\in\R^n$ be the standard basis vectors and let ${\bf 1}\in \R^n$ denote the all-ones vector.
A $t$-linear form $A$ on $\R^n$ is \emph{plane sub-stochastic} if 
$A$ is nonnegative on the standard basis vectors and if for every~$s\in[n]$, we have 
\begin{align}
A(e_s, {\bf 1},{\bf 1},\dots,{\bf 1}) &\leq 1 \nonumber\\
A({\bf 1}, e_s,{\bf 1},\dots,{\bf 1}) &\leq 1 \nonumber\\
&\:\:\vdots \nonumber\\
A({\bf 1}, {\bf 1},\dots,{\bf 1},e_s) &\leq 1. \label{eq:stoch}
\end{align}
Let~$\Pi_n^{(t)}$ be the polytope of $t$-linear forms~$A$ on $\R^n$ such that the form~$|A|$ defined by
$
|A|(e_{s_1},\dots,e_{s_t}) = |A(e_{s_1},\dots,e_{s_t})|,
$
for~$s_1,\dots,s_t\in [n]$,
is plane sub-stochastic.
Observe that the set $\Pi_n^{(2)}$ is the set of matrices $(a_{ij})_{i,j=1}^n$ such that $(|a_{ij}|)_{i,j=1}^n$ is doubly sub-stochastic.\footnote{Recall that the Birkhoff--von~Neumann Theorem states that the Birkhoff polytope is the convex hull of the set of $n\times n$ permutation matrices.
In~\cite{Linial:2014} it is shown that for~$t \geq 3$, the natural analogue of this \emph{fails} for the set of forms in~$\Pi_n^{(t)}$ that attain equalities in~\eqref{eq:stoch} and are nonnegative on standard basis vectors.
}
Our deviation bound then is as follows.

\begin{theorem}\label{thm:tensor-hoeffding}
For every integer~$t \geq 3$ there exist absolute constants $c,C\in (0,\infty)$ such that the following holds.
Let~$A_1,\dots,A_k$ be independent random elements over~$\Pi_n^{(t)}$.
Then, for any $p \geq 1$ and~$\eps > 0$,
\beq\label{eq:tensor-hoeffding}
\Pr\Big[\Big\| \frac{1}{k}\sum_{i=1}^k (A_i - \Exp[A_i])\Big\|_\pnorm > \eps \Big]
\leq
C\exp\Big(-\frac{ck\eps^2}{\sigma_{p,t}(n)^2}\Big),
\eeq
where
\beqn
\sigma_{p,t}(n) = n^{\frac12 - \frac1p}\, \max\{1, n^{1 - \frac{1}{2t}-\frac{t-1}{p}}\}\, (\log n)^{t + \frac12}.
\eeqn
\end{theorem}

For example, for $t \geq 3$, we have $\sigma_{2, t}(n)  = (\log n)^{t + \frac12}$, $\sigma_{t,t}(n) = n^{\frac12 - \frac{1}{2t}}(\log n)^{t + \frac12}$ and $\sigma_{\infty, t}(n) = n^{\frac32 - \frac{1}{2t}}(\log n)^{t + \frac12}$.
The proof of Theorem~\ref{thm:hypar} is now nearly identical to the proof of the Alon--Roichman theorem shown above.

\begin{proof}[ of Theorem~\ref{thm:hypar}]
Let $H = \cay^{(t)}(\Gamma, q,S')$.
For~${\bf g}\in \Gamma^{t}$ let~$A_{\bf g}:\R^\Gamma\times\cdots\times\R^\Gamma\to\R$ be the adjacency form of $\cay(\Gamma,q, \{{\bf g}\})$ and recall from Section~\ref{sec:hypcal} that $A_{\bf g}$ is plane sub-stochastic.
Observe that if~${\bf g}$ is uniform over~$S$, then~$\Exp[A_{\bf g}] = A_K$.
Finally, since $A_H = (A_{{\bf g}_1} + \cdots + A_{{\bf g}_k})/k$, the result follows from Theorem~\ref{thm:tensor-hoeffding} (with $p = t$) and the definition of~$\lambda_K(H)$.
\end{proof}

\begin{remark}[Sub-optimality of Theorem~\ref{thm:tensor-hoeffding}?]\label{rem:ldc}
We conjecture that when $p = t$, the dependence of~\eqref{eq:tensor-hoeffding} on~$n$ is sub-optimal in the sense that $\sigma_{t,t}$ can be replaced with some function $f(n) \leq o(n^{\frac12 - \frac{1}{2t}})$.
However, due to a result of Naor, Regev, and the first author~\cite{Briet:2012d} (see also~\cite{Briet:2015}), it must be the case that $f(n) \geq (\log n)^c$ for every $c > 1$.
Their result implies that for every $t\geq 3$ there exist $\eps(t)\in (0,1)$, $c(t) > 1$ such that the following holds.
For infinitely many~$n\in\N$, there exists a collection of $k = 2^{(\log\log n)^{c(t)}}$ forms $B_1,\dots,B_k\in \Pi_n^{(t)}$ such that for independent Rademacher random variables $\epsilon_1,\dots,\epsilon_k$ (satisfying $\Pr[\epsilon_i = +1] = \Pr[\epsilon_i=-1] = 1/2$), we have
\beqn
\Exp\Big[\Big\|\frac{1}{k}\sum_{i=1}^k\epsilon_iB_i\Big\|_\tnorm\Big] \geq \eps(t).
\eeqn
Setting $A_i = \epsilon_iB_i$, a standard calculation shows that the above expectation is at most
\beqn
\int_0^\infty \Pr\Big[\Big\| \frac{1}{k}\sum_{i=1}^k (A_i - \Exp[A_i])\Big\|_\tnorm > \eps \Big]\, d\eps
\leq
C\sqrt{\frac{f(n)}{k}}
\eeqn
for some absolute constant~$C\in (0,\infty)$, showing that~$f$ cannot be poly-logarithmic in~$n$.
\end{remark}

\subsection*{Open problems}
Our results leave open the problem of determining the minimal degree required for spectral expansion of random Cayley hypergraphs.
Remark~\ref{rem:ldc} could be interpreted as suggesting the intriguing possibility that, in stark contrast with the Alon--Roichman Theorem, this degree must be quasi-polynomial in the size of the group.
Another problem is to determine the optimal form of Theorem~\ref{thm:tensor-hoeffding}.
Finally, it is open if the straightforward generalization of the Expander Mixing Lemma given in Proposition~\ref{prop:genEML} below admits a converse for Cayley hypergraphs.
A converse was shown to hold for Cayley graphs by Kohayakawa, R\"odl, and Schacht~\cite{Kohayakawa:2016} and Conlon and Zhao~\cite{Conlon:2016}.

\subsection*{Acknowledgements}
J.~B. would like to thank Jozef Skokan for pointing him to~\cite{Lenz:2015} and Zeev Dvir and Sivakanth Gopi for helpful discussions.

\section{Proof of Theorem~\ref{thm:apexp}}

In this section we prove Theorem~\ref{thm:apexp}.
To rephrase this result, consider for a set $D\subseteq \F_p^n$ the Cayley hypergraph 
$$L_D = \cay^{(p)}\big(\F_p^n, {\bf 1}, \{\{0,y,2y,\dots, (p-1)y\}\st y\in D\}\big).
$$
Then, by Definition~\ref{def:arithexp}, Theorem~\ref{thm:apexp} says that for every $D\subseteq \F_p^n$ of size $|D| < \delta(p)n^{p-1}$, the hypergraphs $L_D$ and $L_{\F_p^n}$ satisfy $\lambda_{L_{\F_p^n}}(L_D) \geq \eps(p)$.
The first ingredient of the proof is the following straightforward generalization of the Expander Mixing Lemma~\cite{Alon:1988}, which follows directly from the above definitions.

\begin{proposition}[Generalized Expander Mixing Lemma]\label{prop:genEML}
Let $K = \cay^{(t)}(\Gamma, q,S)$ be a Cayley hypergraph, $S'\subseteq S$ be a multiset and $H = \cay^{(t)}(\Gamma, q,S')$.
Then, for every $T_1,\dots,T_t\subseteq \Gamma$,
\beqn
\big|A_H(1_{T_1},\dots,1_{T_t}) - A_K(1_{T_1},\dots,1_{T_t})\big| \leq \lambda_K(H)(|T_t|\cdots |T_t|)^{1/t}.
\eeqn
\end{proposition}

To prove the theorem it thus suffices to show that for every $D\subseteq \F_p^n$ of size $|D|< \delta(p)n^{p-1}$, there exist $T_1,\dots,T_p\subseteq \F_p^n$ such that on the one hand, $A_{L_D}(1_{T_1},\dots,1_{T_p}) = 0$, while on the other hand,   $A_{L_{\F_p^n}}(1_{T_1},\dots,1_{T_p})  \geq \eps(p)(|T_1|\cdots|T_p|)^{1/p}$, which is precisely what we shall do with sets satisfying $T_2 = T_3 = \dots = T_p$.
We achieve this by constructing a combinatorial rectangle $R = T_1\times\cdots\times T_p$ that contains many lines, but no lines with direction in~$D$, by which we mean the following.
Define the \emph{line through $x\in\F_p^n$ in direction~$d\in \F_p^n$}, denoted~$\ell_{x,d}$,  to be the sequence $(x + \lambda d)_{\lambda \in \F_p}$.
Say that $R$
\emph{contains} $\ell_{x,d}$ if $x + \lambda d\in T_{\lambda+1}$ for every $\lambda\in\F$.
Denote by $\mathcal L_D(R)$ the number of lines contained in~$R$ that have direction~$y\in D$.
The following proposition shows why considering lines through rectangles suffices.

\begin{proposition}\label{prop:LDlines}
Let $D,T_1,T_2\subseteq \F_p^n$ so that $T_1$ and $T_2$ are disjoint, and let $R$ be the $p$-dimensional combinatorial rectangle $T_1\times T_2\times\cdots\times T_2$.
Then,
$$
A_{L_D}(1_{T_1},1_{T_2},\dots,1_{T_2})
=
\mathcal L_D(R)/|D|.$$
\end{proposition}

\begin{proof}
Recall from Section~\ref{sec:hypcal} 
and multi-linearity, that
\begin{align}\label{eq:AD}
A_{L_D}(1_{T_1},1_{T_2}\dots,1_{T_2}) &= \sum_{z_1\in T_1}\cdots \sum_{z_p\in T_p}A_{L_D}(1_{\{z_1\}}, \dots, 1_{\{z_p\}}) \nonumber
 \\
&= \frac{1}{|D|p!}\sum_{y\in D}\sum_{x\in \F_p^n}\sum_{\sigma\in S_p}
1_{T_1}\big(x + (\sigma(1) - 1)y\big)
\cdots
1_{T_2}\big(x + (\sigma(p) - 1)y\big).
\end{align}
Consider a pair $x\in\F_p^n, y\in D$ such that the corresponding sum over~$\sigma\in S_p$ in~\eqref{eq:AD} is nonzero.
We claim that in this case, the sum equals~$(p-1)!$.
Indeed, if~$\sigma$ is a permutation such that the corresponding term in the sum equals~1, then since~$T_1$ and~$T_2$ are disjoint, a term corresponding to another permutation~$\sigma'$ is nonzero if and only if  $\sigma'(1) = \sigma(1)$.
Let $P\subseteq \F_p^n\times D$ be the set of such pairs for which the sum over~$\sigma$ is nonzero.  It follows that~\eqref{eq:AD} is equal to ${|P|}/{|D|p}$
and the lemma follows if $|P| = p\mathcal L_D(R)$.  

We compute the size of $P$.
Let $\phi:P\to\mathcal L_D(R)$ be the function $\phi((x, y)) = \ell_{x+(\sigma(1)-1)y, y}$ that maps a pair in $P$ to a line in $R$ where $\sigma$ is an arbitrary permutation that contributes to the corresponding sum in~\eqref{eq:AD}. 
To see that the image of $\phi$ contains only lines in $R$, observe that for every pair $(x, y)\in P$, and for $\lambda = \sigma(1)-1$, we have $x + \lambda y\in T_1$ and $x + \lambda' y\in T_2$ for every $\lambda' \in \F_p\smallsetminus \{\lambda\}$.
Moreover, $\phi$ is surjective, since for each line in $\ell_{x,y}\in \mathcal L_D(R)$, we have $(x,y)\in P$ because the term corresponding to the identity permutation  in~\eqref{eq:AD} is nonzero.

Next we show that for each $\ell_{x,y}\in\mathcal L_D(R)$, its pre-image under~$\phi$ has size exactly~$p$, which implies the proposition.
Let $(x', y')$ be a pair in $\phi^{-1}(\ell_{x, y})$.  Then $y' = y$ is fixed, and $x' = x+\lambda y$ for some $\lambda \in \F_p$. 
We claim that all such choices of $(x', y')$ are in $\phi^{-1}(\ell_{x, y})$.  
Indeed, for every $\lambda \in \F_p$, and $\sigma \in S_p$ such that $\sigma(1)-1 = \lambda$, we have that $x'+(\sigma(1)-1)y' \in T_1$ and $x'+(\sigma(i)-1)y' \in T_2$ for all other $i$.  
Therefore $\phi(x', y') = \ell_{x, y}$.
\end{proof}

Theorem~\ref{thm:apexp} will thus follow from the following result.

\begin{theorem}\label{thm:densecap}
Let $n \geq p^2$ and let $D\subseteq \F^n$ be a set of size $|D| < {n + p - 2\choose p-1}$.
Then, there exist disjoint sets $T_1,T_2\subseteq \F^n$ such that the $p$-dimensional  rectangle $T_1\times T_2\times\cdots\times T_2$ contains at least
$p^{2n + p - p^2}$
 lines, but no lines with direction in~$D$.
\end{theorem}

\begin{proof}[ of Theorem~\ref{thm:apexp}]
Let $T_1,T_2\subseteq \F_p^n$ be as in Theorem~\ref{thm:densecap}.
Then, by Proposition~\ref{prop:LDlines}, we have $A_{L_D}(1_{T_1},1_{T_2},\dots,1_{T_2}) = 0$, but
\beqn
A_{L_{\F_p^n}}(1_{T_1},1_{T_2},\dots,1_{T_2}) 
\geq \frac{1}{p^n}p^{2n + p-p^2} \geq \frac{1}{p^{p^2-p}}|T_1|^{1/p}|T_2|^{(p-1)/p}.
\eeqn
The result now follows from Proposition~\ref{prop:genEML}.
\end{proof}

The proof of Theorem~\ref{thm:densecap} uses the polynomial method.
For the remainder of this section let $\F = \F_p$. 
For an $n$-variate polynomial $f\in \F[x_1,\dots,x_n]$ denote $Z(f) = \{x\in\F^n\st f(x) = 0\}$.

\begin{lemma}\label{lem:zero-one}
Let $f\in\F[x_1,\dots,x_n]$ be a homogeneous polynomial of degree at most~$p-1$.
Let $Z = Z(f)$ and let $a\in\F^*$ be such that the set $S = \{x\in\F^n\st f(x) = a\}$ is nonempty.
Then, the $p$-dimensional rectangle $ Z\times  S \times\cdots\times S$ contains no lines with directions $d\in Z$.
\end{lemma}

\begin{proof}
Recall that a Vandermonde matrix is a square matrix of the form
\beqn
\begin{bmatrix}
1 & a_1 & a_1^2 & \cdots & a_1^{d-1}\\
1 & a_2 & a_2^2 & \cdots & a_2^{d-1}\\
1 & a_3 & a_3^2 & \cdots & a_3^{d-1}\\
\vdots & \vdots & \vdots & \ddots &\vdots\\
1 & a_d & a_d^2 & \cdots & a_d^{d-1}
\end{bmatrix}.
\eeqn
We record the well-known and easy fact that if the~$a_i$ are distinct, then the above matrix has a nonzero determinant and therefore full rank.

For a contradiction, suppose there does exist such a line $\ell_{x,d}$ with $x,d\in Z$.
Consider the polynomial $g\in\F[\lambda]$ defined by $g(\lambda) = a^{-1}f(x + \lambda d)$.
Since~$f$ has degree at most~$p-1$, so does~$g$.
Moreover, since $x,d\in Z$ and since~$f$ is homogeneous, the constant term and the coefficient of~$\lambda^{p-1}$ of~$g$ are zero.
Our assumption that $x + \lambda d\in S$ for every $\lambda\in [p-1]$ then implies
\beqn
g(\lambda) = \sum_{i=1}^{p-2}c_i\lambda^i = a^{-1}f(x + \lambda d) = 1,
\quad
\lambda\in[p-1].
\eeqn
Hence, the all-ones vector ${\bf 1}\in \F^d$ lies in the linear span of the vectors $v_i = (1, 2^i, 3^i, \dots, (p-1)^i)$ for $i\in [p-2]$, since ${\bf 1} = c_1v_1 + \cdots + c_{p-1}v_{p-2}$.
But the matrix $[{\bf 1}, v_1,\dots, v_{p-2}]$ is a  full-rank Vandermonde matrix, which is a contradiction.
\end{proof}

The following basic and standard result (see for example~\cite{Tao:2014}) shows that for any small set $D\subseteq\F^n$, we can always find a low-degree homogeneous polynomial~$f$ such that $D\subseteq Z(f)$.

\begin{lemma}[Homogeneous Interpolation]\label{lem:hom-interpolation}
For every $D\subseteq \F^n$ of size $|D| < {n+d-1\choose d}$
there exists a nonzero homogeneous polynomial $f\in\F[x_1,\dots,x_n]$ of degree~$d$ such that $D\subseteq Z(f)$.
\end{lemma}

\begin{proof}
Let~$V$ be the vector space of homogeneous degree-$d$ polynomials in~$\F[x_1,\dots,x_n]$.
Note that
 $\dim(V) = {n+d - 1\choose n-1}$.
 Let~$W = \F^D$.
Let $L:V\to W$ be the linear map given by $L(f) = (f(x))_{x\in D}$.
Since $\dim(W) < \dim(V)$, it follows from the Rank Nullity Theorem that $\dim(\ker(L)) \geq 1$.
Hence, there exists a nonzero $f\in V$ such that $f(x) = 0$ for all $x\in D$.
\end{proof}

We also use the following standard result bounding the zero-set of a polynomial in terms of its degree; the specific form quoted below is from~~\cite[Lemma~2.2]{CohenTal:2014}.

\begin{lemma}[DeMillo--Lipton--Schwartz--Zippel]\label{lem:DMLSZ}
Let~$\F$ be a finite field with~$q$ elements
and
let~$f\in\F[x_1,\dots,x_n]$ be a nonzero polynomial of degree~$d$.
Then,
\beqn
|Z(f)| \leq \Big(1 - \frac{1}{q^{d/(q-1)}}\Big)q^n.
\eeqn
\end{lemma}

\begin{proof}
For $g\in\F[x_1,\dots,x_m]$ write $\overline{Z(g)} = \F^m\smallsetminus Z(g)$.
Using induction on~$n$ we shall prove that $|\overline{Z(f)}| \geq q^{n}/q^{d/(q-1)}$, which establishes the result.
First observe that, by Lagrange's Theorem, we may assume that each variable in~$f$ has degree at most~$q - 1$.
The base case $n = 1$ follows from the Factor Theorem, since then $|\overline{Z(f)}| \geq q - d \geq q(1 - d/q) \geq q/q^{d/(q-1)}$.

Assume the result holds for $(n - 1)$-variable polynomials.
We can decompose~$f$ as
\beq\label{eq:fdecomp}
f(t,y_1,\dots,y_{n-1}) = \sum_{i=1}^{\min\{d, q-1\}}t^i g_i(y_1,\dots,y_{n-1}),
\eeq
where $g_i\in\F[y_1,\dots,y_{n-1}]$ has degree at most~$d - i$.
Let~$k$ be the maximum~$i$ for which~$g_i$ is nonzero.
By the induction hypothesis, the polynomial~$g_k$  satisfies 
$
|\overline{Z(g_k)}| \geq q^{n-1}/q^{(d-k)/(q-1)}$.

For each~$y\in \overline{Z(g_k)}$ let $h_y\in \F[t]$ be the univariate polynomial defined by $h_y(t) = f(t, y_1,\dots,y_{n-1})$.
The decomposition~\eqref{eq:fdecomp} shows that each~$h_y$ is nonzero and has degree~$k$, and thus $\overline{Z(h_y)} \geq q/q^{k/(q-1)}$.
We conclude that
\beqn
|\overline{Z(f)}|
\geq
\sum_{y \in \overline{Z(g_k)}}
|\overline{Z(h_y)}|
\geq
q^{n}/q^{d/(q-1)}.
\eeqn
\end{proof}

Finally, we use the Chevalley--Warning Theorem to lower bound the number of common zeros of a system of polynomials~\cite[Chapter~6]{Lidl:1983}.

\begin{theorem}[Chevalley--Warning]\label{thm:CW2}
Let $\F$ be a finite field and let $f_1,\dots,f_k\in\F[x_1,\dots,x_n]$ be nonzero polynomials such that
$
d = \deg(f_1) + \cdots + \deg(f_k) < n$.
If there is at least one solution to the system $f_1(x) = \cdots = f_k(x) = 0$ in $\F^n$, then there are at least~$|\F|^{n-d}$ solutions.
\end{theorem}

We include a quick proof we learned from Dion Gijswijt, which is based on Lemma~\ref{lem:DMLSZ}.

\begin{proof}
Define the polynomial $f\in \F[x_1,\dots,x_n]$ by
$f = (1 - f_1^{q-1})\cdots (1 - f_k^{q-1})$.
Observe that $\deg(f) \leq (q-1)d$ and that $f(x) =1$ if $x\in Z(f_1,\dots,f_k)$ and $f(x) = 0$ otherwise.
By Lemma~\ref{lem:DMLSZ}, 
$
|Z(f)| \leq (1 - 1/q^{d})q^n
$.
Hence, $|Z(f_1,\dots,f_k)| \geq q^{n - d}$.
\end{proof}

With this, we are set up to prove Theorem~\ref{thm:densecap}.

\begin{proof}[ of Theorem~\ref{thm:densecap}]
By Lemma~\ref{lem:hom-interpolation}, there exists a nonzero degree-$(p-1)$ homogeneous polynomial $f\in\F[x_1,\dots, x_n]$ such that $D\subseteq Z(f)$.
Set $T_1 = Z(f)$.
By Lemma~\ref{lem:DMLSZ}, there exists an~$a\in \F^*$ such that the set~$S = \{x\in\F^n\st f(x) = a\}$ is nonempty.
For each $\lambda\in[p-1]$ set $T_{\lambda + 1} = S$.
It then follows from Lemma~\ref{lem:zero-one} that the combinatorial rectangle $R = T_1 \times\cdots\times T_{p}$ contains no lines with direction in~$D$.

We show that~$R$ contains many lines.
To this end, define degree-$(p-1)$ polynomials $g_0,\dots, g_{p-1}\in \F[x_1,\dots, x_n,y_1,\dots,y_n]$ by setting $g_0(x,y) = f(x)$ and $g_{\lambda} = f(x + \lambda y) - a$ for each $\lambda \in [p-1]$.
Then, a solution in~$\F^{2n}$ to the set of equations $g_0(x,y) = 0,\dots, g_{p-1}(x,y) = 0$ is a line through~$R$.
There is at least one such solution.
Indeed, if we let $x = 0$ and $y\in S$, then
since~$f$ is homogenous of degree~$p-1$, we have $g_0(0,y) = f(0) = 0$ and $g_\lambda(0,y) = f(\lambda y) -a = a(\lambda^{p-1} - 1) = 0$ by Fermat's Little Theorem.
By Theorem~\ref{thm:CW2}, the above system has at least $p^{2n +p - p^2}$ solutions in~$\F^{2n}$ and~$R$ has at least that many lines.
\end{proof}

\section{Proof of Theorem~\ref{thm:tensor-hoeffding}}

In this section we prove Theorem~\ref{thm:tensor-hoeffding}.
Throughout this section, let $(\epsilon_i)_{i\in\N}$ be independent uniformly distributed $\pmset{}$-valued random variables and let~$\epsilon = (\epsilon_1,\dots,\epsilon_k)$.

\subsection{Reduction to Bernoulli processes}

The main new ingredient needed for the proof of Theorem~\ref{thm:tensor-hoeffding} is a bound showing that for fixed $A_1,\dots, A_k\in \Pi_n^{(t)}$, the expected norm of the Rademacher sum $\epsilon_1A_1 + \cdots + \epsilon_kA_k$ is at most a constant times $\sqrt{k}\,\sigma_{p,t}(n)$.
From this, we derive the result using standard techniques based on combining a symmetrization trick, the Kahane--Khintchine inequality and an exponential Markov inequality.
The details follow below.
Recall that a real-valued random variable is \emph{centered} if it has expectation zero.

The following standard symmetrization lemma allows us to bound the moments of the random variable whose tail we aim to bound in~\eqref{eq:tensor-hoeffding} in terms of the moments of the norm of a Rademacher sum of \emph{fixed} plane sub-stochastic forms.

\begin{lemma}[Symmetrization]\label{lem:symmetrization}
Let~$X$ be a real finite-dimensional normed vector space and let~$Y_1,\dots,Y_k\subseteq X$ be subsets.
For $p\geq 1$ let $\sigma_p\in [0,\infty)$ be the smallest number such that for any fixed $A_1\in Y_1,\dots,A_k\in Y_k$, we have
\beqn
\Big(\Exp\Big[\Big\|\sum_{i=1}^k \epsilon_iA_i\Big\|^p\Big]\Big)^{1/p}
\leq\sigma_p.
\eeqn
Then, if $A_1,\dots,A_k$ are independent random variables over~$Y_1,\dots,Y_k$, respectively, 
\beq\label{eq:symm}
\Big(\Exp\Big[\Big\| \sum_{i=1}^k (A_i - \Exp[A_i]) \Big\|^p\Big]\Big)^{1/p}
\leq
4\sigma_p.
\eeq
\end{lemma}

\begin{proof}
For each~$i\in[k]$ let $A_i'$ be an independent copy of~$A_i$. 
Let $B_i = A_i - \Exp[A_i']$, let~$\widetilde B_i$ be an independent copy of~$B_i$ and note that these random variables are centered.
By Jensen's inequality and the triangle inequality, the $p$th power of the left-hand side of~\eqref{eq:symm} is at most
\begin{align*}
\Exp\Big[\Big\|\sum_{i=1}^kB_i\Big\|^p\Big]
=
\Exp\Big[\Big\|\sum_{i=1}^kB_i - \Exp\Big[\sum_{j=1}^k\widetilde B_j\Big]\Big\|^p\Big]
\leq
\Exp\Big[\Big\|\sum_{i=1}^k(B_i - \widetilde B_i)\Big\|^p\Big].
\end{align*}
Since the random variables~$B_i - \widetilde B_i$ are independent and symmetrically distributed, that is, $B_i - \widetilde B_i$ has the same distribution as~$\widetilde B_i - B_i$, it follows that their sum has the same distribution as $\delta_1(B_1 - \widetilde B_1) + \cdots + \delta_k(B_k - \widetilde B_k)$ for any choice of signs $\delta_i\in \pmset{}$.
Hence, by Jensen's inequality, the above is at most
\begin{align*}
\Exp_{\epsilon}\Big[\Exp_{B_i,\widetilde B_i}\Big[\Big\|\sum_{i=1}^k\epsilon_i(B_i - \widetilde B_i)\Big\|^p\Big]\Big]
&\leq
2^p\Exp_{\epsilon}\Big[\Exp_{B_i}\Big[\Big\|\sum_{i=1}^k\epsilon_iB_i\Big\|^p\Big]\Big]\\
&=
2^p\Exp_{\epsilon}\Big[\Exp_{A_i,A_i'}\Big[\Big\|\sum_{i=1}^k\epsilon_i(A_i - \Exp[A_i'])\Big\|^p\Big]\Big].
\end{align*}
Independence of~$A_i$ and~$A_i'$ for each~$i\in[k]$, another application of Jensen's inequality and the triangle inequality imply that the above is at most
\beqn
4^p\Exp_{\epsilon}\Big[\Exp_{A_i}\Big[\Big\|\sum_{i=1}^k\epsilon_i A_i\Big\|^p\Big]\Big]
=
4^p\Exp_{A_i}\Big[\Exp_{\epsilon}\Big[\Big\|\sum_{i=1}^k\epsilon_i A_i\Big\|^p\Big]\Big].
\eeqn
The result now follows by applying the definition of~$\sigma_p$ to the inner expectation above.
\end{proof}

Next, the Kahane--Khintchine inequality reduces the problem of bounding the  numbers~$\sigma_p$ from Lemma~\ref{lem:symmetrization} to bounding~$\sigma_1$ only (see for example~\cite[Theorem 4.7]{Ledoux:1991}).

\begin{theorem}[Kahane--Khintchine inequality]\label{thm:Kahane-Khintchine}
Let $X$ be a Banach space and ${A_1,\dots,A_k\in X}$.
Then, for any integer~$p\geq 1$ and some absolute constant $C$, we have
\beq\label{eq:KK}
\Big(\Exp\Big[\Big\|\sum_{i=1}^k \epsilon_iA_i\Big\|^p\Big]\Big)^{1/p}
\leq
\sqrt{Cp}\,
\Exp\Big[\Big\|\sum_{i=1}^k \epsilon_iA_i\Big\|\Big].
\eeq
\end{theorem}

Lemma~\ref{lem:symmetrization} and Theorem~\ref{thm:Kahane-Khintchine} thus show that the moments on the left-hand side of~\eqref{eq:symm} can be bounded in terms of the average on the right-hand side of~\eqref{eq:KK}.
The following upper bound and a standard exponential Markov argument will now allow us to prove Theorem~\ref{thm:tensor-hoeffding}.

%
%
%
%
%

%

\begin{theorem}\label{thm:process}
For every integer~$t\geq 3$ there exists an absolute constant $C(t) \in (0,\infty)$ such that the following holds.
Let~${A_1,\dots,A_k\in \Pi_n^{(t)}}$.
Then, for any $p\geq 1$ and $\sigma_{p,t}(n)$ as in Theorem~\ref{thm:tensor-hoeffding}, 
\beq\label{eq:process-val}
\Exp\Big[
\Big\|\sum_{i=1}^k\epsilon_iA_i\Big\|_\pnorm
\Big]
\leq
C(t)\sqrt{k}\, \sigma_{p,t}(n)
\eeq
\end{theorem}

\begin{proof}[ of Theorem~\ref{thm:tensor-hoeffding}]
Define the random variable
\beqn
Z = \Big\| \frac{1}{k}\sum_{i=1}^k (A_i - \Exp[A_i]) \Big\|_\pnorm.
\eeqn
Let $\alpha > 0$ be a parameter to be set later.
Then, by Markov's inequality,
\begin{align}
\Pr[Z > \eps]
=
\Pr\big[e^{\alpha Z^2} > e^{\alpha \eps^2}\big]
\leq
e^{-\alpha \eps^2}\Exp\big[e^{\alpha Z^2}\big].\label{eq:markov}
\end{align}

Lemma~\ref{lem:symmetrization}, Theorem~\ref{thm:Kahane-Khintchine} and Theorem~\ref{thm:process} imply that for every integer~$p\geq 1$, we have
\begin{align*}
\big(\Exp[Z^p]\big)^{1/p}
&\leq
4C\sqrt{p}\,
C(t)\frac{\sigma_{p,t}(n)}{\sqrt k}.
\end{align*}
Let $\sigma = CC(t)\sigma_{p,t}(n)\sqrt{16/k}$, so that the above equals $\sqrt{p}\, \sigma$.
It follows that
\begin{align*}
\Exp\big[e^{\alpha Z^2}\big]
&=
1 + \sum_{p=1}^\infty \frac{\alpha^p\Exp[Z^{2p}]}{p!}\\
&\leq
1 + \sum_{p=1}^\infty \frac{\alpha^p(2p)^p\sigma^{2p}}{p!}\\
&\leq
1 + \sum_{p=1}^{\infty}\Big(\frac{2 \alpha \sigma^2}{e}\Big)^p,
\end{align*}
where in the last line we used that $p! \geq (p/e)^p$.
Set $\alpha = e/(4\sigma^2)$. Then, the above geometric series equals~2 and the right-hand side of~\eqref{eq:markov} is at most $2e^{-\alpha\eps^2}$, giving the result.
\end{proof}

The remainder of this section is devoted to the proof of Theorem~\ref{thm:process}.

\subsection{Dyadic decomposition}

The first step towards proving Theorem~\ref{thm:process} is to break the problem up into more manageable pieces using the following lemma.
For every~$d\in[n]$ define the set
\beq\label{eq:Hball}
\Cball_d^n = \big\{x\in \{-1,0,1\}^n \st \|x\|_{\ell_0} = \min\{d,n\}\big\}.
\eeq

\begin{lemma}\label{lem:dyadic}
Let $R = \ceil{\log n}$. 
Then for $p \geq 1$, 
\beq\label{eq:dyadic}
\Exp\Big[
\Big\|\sum_{i=1}^k\epsilon_iA_i\Big\|_\pnorm
\Big]
\leq
2^t\sum_{{\bf r} \in [R]^t} \frac{\Exp\Big[
\max\Big\{\Big(\sum_{i=1}^k\epsilon_i A_i\Big)({\bf x})
\st
{\bf x} \in \Cball_{2^{r_1}}^n\times\cdots\times\Cball_{2^{r_t}}^n\Big\}
\Big]}{2^{\tfrac{r_1+\cdots+r_t}{p}}}
.
\eeq
\end{lemma}

\begin{proof}
Partition the unit ball~$B_{p}^n$ of $\ell_p^n$ into $R$ pieces defined for each $r\in [R]$ by
\begin{align*}\label{eq:Spdef}
\mathcal S_t(r) &= \big([-\tfrac{2}{2^{r/p}}, -\tfrac{1}{2^{r/p}}\big)\cup \{0\} \cup \big(\tfrac{1}{2^{r/p}}, \tfrac{2}{2^{r/p}}]\big)^n\cap B_{p}^n
\quad\quad
\text{and}
\quad\quad
\mathcal S_t(R) = [-\tfrac{1}{2^{R/p}}, \tfrac{1}{2^{R/p}}]^n.
\end{align*}
Then, since each~$A_i$ is linear in each of its arguments, 
\begin{align*}
\Exp\Big[
\Big\|\sum_{i=1}^k\epsilon_iA_i\Big\|_\pnorm
\Big]
&=
\Exp\Big[
\sup\Big\{
\Big(\sum_{i=1}^k\epsilon_iA_i\Big)({\bf x})
\st
{\bf x} \in B_{p}^n\times\cdots\times B_{p}^n
\Big\}
\Big]
\\
&\leq
\sum_{{\bf r} \in [R]^t}
\Exp\Big[
\sup\Big\{
\Big(\sum_{i=1}^k\epsilon_iA_i\Big)({\bf x})
\st
{\bf x} \in \mathcal S_{p}(r_1)\times\cdots\times \mathcal S_{p}(r_t)
\Big\}
\Big].
\end{align*}

Note that any~$x\in \mathcal S_p(r)$ has at most~$2^{r}$ nonzero entries and that those entries have magnitude at most $2/2^{r/p}$.
Hence, by multi-linearity of the~$A_i$ and convexity, the above suprema are bounded from above by
\beqn
\frac{2^t}{2^{\frac{r_1 + \dots +r_t}{p}}}
\max\Big\{\Big(\sum_{i=1}^k\epsilon_i A_i\Big)({\bf x})
\st
{\bf x} \in \Cball_{2^{r_1}}^n\times\cdots\times\Cball_{2^{r_t}}^n\Big\}.\qedhere
\eeqn
\end{proof}

The above lemma thus reduces the problem of bounding the expectations of Theorem~\ref{thm:process} to bounding each of the expectations appearing in the right-hand side of~\eqref{eq:dyadic}.
The following lemma provides the bounds we need.

\begin{lemma}\label{lem:smallprocess}
Let ${\bf d} \in [n]^t$.
Let $\min({\bf d}) = \min\{d_1,\dots,d_t\}$ and let $\max({\bf d}) = \max\{d_1,\dots, d_t\}$.
Then, for $\Cball_{d}^n$ as in~\eqref{eq:Hball}, we have
\beqn
\Exp \Big[\max\Big\{\Big(\sum_{i=1}^k\epsilon_iA_i({\bf x}) \st x[s] \in \Cball_{d_s}^n,\, s\in[t]\Big\}\Big]
\leq
C(t)\sqrt{k\max({\bf d})\log n}\, \min({\bf d})^{1 - \frac{1}{2t}},
\eeqn
where $C(t)\in (0,\infty)$ depends on~$t$ only.
\end{lemma}

\begin{proof}[ of Theorem~\ref{thm:process}]
Combining Lemmas~\ref{lem:dyadic} and~\ref{lem:smallprocess}  shows that the left-hand side of~\eqref{eq:process-val} is at most
\beqn
C''\sqrt{k\log n}\sum_{r_1,\dots,r_t=1}^{\log n} \frac{\sqrt{\max_s(2^{r_s})}\min_s(2^{r_s})^{1 - 1/(2t)}}{2^{(r_1 + \cdots + r_t)/p}}.
\eeqn

We claim that each of the above fractions is at most $n^{\frac12 - \frac1p}\max\{1, n^{1 - \frac{1}{2t}-\frac{t-1}{p}}\}$, from which the claim follows.
Indeed, considering the square of these fractions, for any ${\bf d} = (d_1,\dots,d_t)$ such that $d_1 \geq d_2\geq \cdots \geq d_t$, we have
\begin{align*}
\frac{\max({\bf d})\min({\bf d})^{2 - 1/t}}{(d_1\cdots d_t)^{2/p}}
&=
\frac{d_1 d_t^{2 - 1/t}}{(d_1\cdots d_t)^{2/p}}\\
&\leq
d_1^{1 - 2/p} \frac{d_t^{2 - 1/t}}{d_t^{2(t-1)/p}}\\
&=
d_1^{1 - 2/p} \max\{1, d_t^{2-1/t-2t/p+2/p}\}\\
&\leq
n^{1 - 2/p}\, \max\{1, n^{2 - 1/t-2(t-1)/p}\},
\end{align*}
as claimed.
\end{proof}

\subsection{Dudley's integral inequality}

To prove Lemma~\ref{lem:smallprocess} we use Dudley's \emph{integral inequality} (see for example~\cite[Lemma~2.2.1 and Eq.~(2.38)]{Talagrand:2014}), 
which bounds the expected supremum of a stochastic process endowed with a metric space structure in terms of covering numbers.
For a metric space $(\Lambda, d)$ and $\eps > 0$, an \emph{$\eps$-net} is a subset $\Lambda'\subseteq \Lambda$ such that for every~$\lambda\in \Lambda$ there exists an $\gamma\in\Lambda'$ with distance $d(\lambda,\gamma)\leq \eps$
and the \emph{covering number} $N(\Lambda,d,\eps)$ is the smallest integer~$N$ such that $(\Lambda,d)$ admits an~$\eps$-net of size~$N$.
The \emph{diameter} of a metric space $(\Lambda, d)$ is given by $\diam(\Lambda) = \sup\{d(\lambda,\gamma)\st \lambda,\gamma\in \Lambda\}$.

\begin{theorem}[Dudley's integral inequality]
\label{thm:dudley}
There exists an absolute constant~$C\in (0,\infty)$ such that the following holds.
Let~$\Lambda$ be a finite set and $d:\Lambda\times\Lambda\to\R_+$ be a metric on~$\Lambda$.
Let $(X_\lambda)_{\lambda\in \Lambda}$ be a collection of centered random variables such that for every $\lambda,\gamma\in\Lambda$ and any~$\eps > 0$, we have
\beqn
\Pr\big[|X_{\lambda} - X_{\gamma}| > \eps\big]
\leq
2\exp\Big(-\frac{\eps^2}{d(\lambda,\gamma)^2}\Big).
\eeqn
Then,
\beqn
\Exp\big[\max_{\lambda\in \Lambda} X_\lambda\big]
\leq
C\int_0^{\diam(\Lambda)}\sqrt{\log N(\Lambda, d, \eps)}\, d\eps.
\eeqn
\end{theorem}

The following set is relevant to our setting:
\beq\label{def:Lambda}
\Lambda_{\bf d}
=
\big\{\big(A_1({\bf x}),\dots, A_k({\bf x})\big)
\st
x[s] \in \Cball_{d_s}^n,\, s\in[t]
\big\}
\subseteq\R^k.
\eeq
For each $\lambda \in \Lambda_{\bf d}$ consider the (centered) random variable $X_\lambda = \langle \epsilon,\lambda\rangle$, so that the left-hand side of~\eqref{eq:process-val} equals $\Exp[\max_{\lambda\in\Lambda_{\bf d}}X_\lambda]$.
Moreover, for every $\lambda,\gamma\in\Lambda_{\bf d}$ and $\eps > 0$, we have
\begin{align}
\Pr\big[|X_\lambda - X_\gamma| > \eps\big] &=
\Pr\Big[\Big| \sum_{i=1}^k\epsilon_i (\lambda_i - \gamma_i) \Big| > \eps\Big]
\leq
2\exp\Big(-\frac{\eps^2}{2\|\lambda - \gamma\|_{\ell_2}^2}\Big),\label{eq:increment}
\end{align}
where the second line follows from Hoeffding's inequality~\cite[Theorem~2.8]{Boucheron:2013}.
We shall therefore consider the metric space $(\Lambda_{\bf d}, \ell_2)$.
For our setting, the relevant form of Dudley's inequality is then as follows.

\begin{corollary}\label{cor:dudley-lambda}
There exists an absolute constant $C\in(0,\infty)$ such that the following holds.
Let~$\Lambda_{\bf d}\subseteq \R^k$ be as in~\eqref{def:Lambda}.
Then,
\beq\label{eq:dudley}
\Exp\big[\max_{\lambda\in\Lambda_{\bf d}}\langle \epsilon, \lambda\rangle\big] \leq C\int_0^{\diam(\Lambda_{\bf d})}\sqrt{\log N(\Lambda_{\bf d}, \ell_2, \eps)}\, d\eps.
\eeq
\end{corollary}

\begin{proof}
Let $d$ be the metric on~$\Lambda_{\bf d}$ given by $d(\lambda,\gamma) = \sqrt{2}\|\lambda - \gamma\|_{\ell_2}$.
Then, the diameter of~$\Lambda_{\bf d}$ under~$d$ equals $\sqrt{2}\diam(\Lambda_{\bf d})$.
Moreover, since an $\eps$-net for $(\Lambda_{\bf d}, d)$ is an $\eps/\sqrt{2}$-net for $(\Lambda_{\bf d},\ell_2)$ and \emph{vice versa}, we have $N(\Lambda_{\bf d}, d,\eps) = N(\Lambda_{\bf d}, \ell_2, \eps/\sqrt{2})$.
By~\eqref{eq:increment},
 Theorem~\ref{thm:dudley} and a change of variables, the left-hand side of~\eqref{eq:dudley} is therefore at most
\begin{align*}
C\int_0^{\sqrt{2}\diam(\Lambda_{\bf d})} \sqrt{\log N(\Lambda_{\bf d},d,\eps)}\,d\eps
&=
C\int_0^{\sqrt{2}\diam(\Lambda_{\bf d})} \sqrt{\log N(\Lambda_{\bf d},\ell_2,\eps/\sqrt{2})}\,d\eps\\
&=
\sqrt{2}C\int_0^{\diam(\Lambda_{\bf d})} \sqrt{\log N(\Lambda_{\bf d},\ell_2,\delta)}\,d\delta
\end{align*}
\end{proof}

To apply the above result we need a bound on the diameter of~$\Lambda_{\bf d}$ and its covering numbers.

\begin{proposition}\label{prop:diameter}
We have
$\diam(\Lambda_{\bf d}) \leq \min({\bf d})\sqrt{k}$.
\end{proposition}
\begin{proof}
Because $A_i$ is plane sub-stochastic, $|A_i({\bf x})| \leq \min({\bf d})$ for every ${\bf x} \in \Cball$.  Therefore $\Lambda_{\bf d}$ is a set of $k$-dimensional vectors whose entries are bounded above in absolute value by $\min({\bf d})$, and the proposition follows.
\end{proof}

\subsection{Bounds on the covering numbers}

To bound the covering numbers of~$(\Lambda_{\bf d}, \ell_2)$ we use a technique akin to Maurey's empirical method for bounding $N(B_{\ell_1^n},\ell_2,\eps)$, the covering numbers of the unit ball of~$\ell_1^n$ under~$\ell_2$ (see for instance~\cite[Section~1.4]{Chafai:2012}).  
In particular, for every $t$-tuple ${\bf x}$ as in~\eqref{def:Lambda}, we use the probabilistic method to show that there exists another $t$-tuple ${\bf \widetilde x}$ of vectors $\widetilde x[s]\in\R^n$ such that each $\widetilde x[s]$ is a ``sparse'' version of~$x[s]$. By this we mean that it has few nonzero entries, each of which has relatively small magnitude, and such that ${\bf A(\widetilde x)}$ is close to~${\bf A(x)}$ in Euclidean distance. This implies that there exists a net composed of all points ${\bf A(\widetilde x)}$ such that~${\bf \widetilde x}$ is sparse and that the covering numbers can be bounded by the number of $t$-tuples of sparse vectors.
The sparse vectors themselves are obtained by taking the empirical average of independent samples from a signed and scaled standard basis vector whose average equals~$x[s]$.
Quantitatively, we get the following lemma.

\begin{lemma}\label{lem:maurey}
For any $\eps \geq 1$, the set $\Lambda_{\bf d}$ has an $\eps\sqrt{k\min({\bf d})}$-net of size~$n^{2t\max({\bf d})/\eps^{2/t}}$.
\end{lemma}

\begin{proof}
Fix  $x[1]\in \Cball_{d_1}^n,\dots,x[t]\in\Cball_{d_t}^n$ and let $D_1,\dots,D_t\subseteq [n]$ denote their supports.
Let $\eta =\eps^{2/t}$ and do the following for each $s\in [t]$: Set $c_s = d_s/\eta$ (assume for simplicity that this is an integer) and
note that $c_s \leq n$.  
Let $e[s]$ be an independent random standard basis vector of~$\R^n$ whose nonzero coordinate is distributed uniformly over~$D_s$.
For each $l\in [c_s]$ let~$e[s]_l$ be an independent copy of $e[s]$.
Define the random vector $\widetilde x[s]_l = (x[s]\circ e[s]_l)$, where~$\circ$ denotes coordinate-wise multiplication, and define
\beqn
\widetilde x[s]
=
\frac{1}{c_s} \sum_{l=1}^{c_s} d_s\, \widetilde x[s]_l.
\eeqn
Thus, $\widetilde x[s]$ is the empirical average of independent random vectors with expectation~$x[s]$.
By multi-linearity of the~$A_i$ it follows that $\Exp[{\bf A(\widetilde x)}] = {\bf A(x)}$.

We bound the expected Euclidean distance of ${\bf A(\widetilde x)}$ and ${\bf A}({\bf x})$.
By Jensen's inequality,
\begin{align}
\Exp\|{\bf A(\widetilde x) - A(x)}\|_{\ell_2}
\leq
\left(\sum_{i=1}^k\Exp\big[ |A_i({\bf \widetilde x}) - A_i({\bf x})|^2\big]\right)^{1/2} =
\left(\sum_{i=1}^k\var\big[ A_i({\bf \widetilde x})\big]\right)^{1/2}. \label{eq:var1}
\end{align}

We treat each of the above variances separately.  
Fix an $i\in[k]$. 
For each ${\bf l}\in [c_1]\times\cdots\times[c_t]$ define $\widetilde{\bf x}_{\bf l} = \widetilde x[1]_{l_1}\times\cdots\times \widetilde x[t]_{l_{t}}$ and ${\bf e}_{\bf l} = e[1]_{l_1}\times\cdots\times e[t]_{l_{t}}$.  
Using multi-linearity of~$A_i$, we get 
\begin{align*}
\Exp \big[A_i({\bf \widetilde x})^2\big]
&= \Big(\frac{d_1\cdots d_t}{c_1\cdots c_t}\Big)^{2}\sum_{{\bf l,l'}\in[c]^t} \Exp\big[A_i(\widetilde{\bf x}_{\bf l})A_i(\widetilde{\bf x}_{\bf l'})\big].
\end{align*}
For every pair ${\bf l,l'}\in[c_1]\times\cdots\times[c_t]$ denote by $\Delta({\bf l}, {\bf l'})\subseteq[t]$ the set of coordinates $s\in[t]$ such that $l_s\ne l_s'$.  
In the case where $|\Delta({\bf l}, {\bf l'})| = t$, the tuples ${\bf e_l}$ and ${\bf e_{l'}}$ are independent.
It follows that in that case, the random variables $A_i(\widetilde{\bf x}_{\bf l})$ and $A_i(\widetilde{\bf x}_{\bf l'})$ are independent and that 
\beqn
\Exp\big[A_i(\widetilde{\bf x}_{\bf l})A_i(\widetilde{\bf x}_{\bf l'})\big] 
= 
\frac{A_i({\bf x})^2}{(d_1\cdots d_t)^2}.  
\eeqn
Since there are fewer than $(c_1\cdots c_t)^2$ pairs ${\bf l,l'}\in[c_1]\times\cdots\times[c_t]$ such that $|\Delta({\bf l,l'})| = t$, the variance is at most
\begin{align*}
\var\big[
A_i(\widetilde{\bf x})
\big]
&=
\Big(\frac{d_1\cdots d_t}{c_1\cdots c_t}\Big)^{2}\Bigg(
\sum_{|\Delta({\bf l,l'})| < t}\Exp\big[A_i(\widetilde{\bf x}_{\bf l})A_i(\widetilde{\bf x}_{\bf l'})\big]
+
\frac{1}{(d_1\cdots d_t)^2}\sum_{|\Delta({\bf l,l'})| = t}A_i({\bf x})^2
\Bigg)
-
A_i({\bf x})^2\\
&\leq
\eta^{2t}
\sum_{\Delta({\bf l,l'}) < t}\Exp\big[A_i(\widetilde{\bf x}_{\bf l})A_i(\widetilde{\bf x}_{\bf l'})\big].
\end{align*}
Fix a pair ${\bf l,l'}\in[c_1]\times\cdots\times[c_t]$ such that $S = \Delta({\bf l,l'})$ satisfies $|S| = r < t$
and assume for simplicity that $l_s \neq l'_s$ when $s \leq r$.
Since $x[s]\in\Cball$ for each $s\in[t]$ and since~$A_i$ is plane sub-stochastic,
\begin{align*}
\Exp\big[A_i(\widetilde{\bf x}_{\bf l})A_i(\widetilde{\bf x}_{\bf l'})\big]
&=
\Exp\Bigg[\Bigg(\prod_{i=1}^t \langle x[i],e[i]_{l_i}\rangle\langle x[i],e[i]_{l'_i}\rangle\Bigg) A_i({\bf e}_{\bf l})A_i({\bf e}_{\bf l'})\Bigg]\\
&\leq
\Exp\big[|A_i|({\bf e}_{\bf l})|A_i|({\bf e}_{\bf l'})\big]\\
&=
\frac{1}{(d_1\cdots d_r)^{2}}\Exp \big[|A_i|({\bf 1}_{D_1}, \dots, {\bf 1}_{D_r}, e[r+1]_{l_{r+1}}, \dots, e[t]_{l_t})^2\big] \\
&\leq
\frac{1}{(d_1\cdots d_r)^{2}}\Exp \big[|A_i|({\bf 1}_{D_1}, \dots, {\bf 1}_{D_r}, e[r+1]_{l_{r+1}}, \dots, e[t]_{l_t})\big] \\
&\leq
\frac{|A_i|({\bf 1}_{D_1},\dots, {\bf 1}_{D_t})}{(d_1\cdots d_r)^2 d_{r+1}\cdots d_t}\\
&\leq
\frac{\min({\bf d})}{(d_1\cdots d_r)^2 d_{r+1}\cdots d_t}.
\end{align*}
In general,
\beqn
\Exp\big[A_i(\widetilde{\bf x}_{\bf l})A_i(\widetilde{\bf x}_{\bf l'})\big]
\leq
\frac{\min({\bf d})}{\prod_{s\in S}d_s^2\prod_{s\in [t]\smallsetminus S} d_s}.
\eeqn

The number of ${\bf l, l'}\in[c_1]\times\cdots\times[c_t]$ such that $\Delta({\bf l,l'}) = S$ is at most $\prod_{s\in S}c_s^2\prod_{s\in [t]\smallsetminus S}c_s$.
Hence, using the definition of~$c_s$ and the assumption that~$\eta \geq 1$, we  get
\begin{align*}
\var\big[
A_i(\widetilde{\bf x})
\big]
&\leq
\eta^{2t}
\sum_{r=0}^{t-1}\sum_{S\in {[t]\choose r}} \Big(\prod_{s\in S}c_s^2\prod_{s\in [t]\smallsetminus S}c_s\Big)\Big( \frac{\min({\bf d})}{\prod_{s\in S}d_s^2\prod_{s\in [t]\smallsetminus S} d_s}\Big)\\
&\leq
\eta^{t}\min({\bf d}) \sum_{r=0}^{t-1} {t\choose r}\eta^{-r}\\
&\leq 2^t \eta^{t}\min({\bf d}).
\end{align*}
Plugging this into~\eqref{eq:var1}, it then follows
from the Averaging Principle, there exist vectors $y[s]\in (d_s/c_s)\{-c_s,-c_s+1,\dots,0,\dots,c_s-1,c_s\}$ with at most~$c_s$ nonzero entries such that
\beqn
\|{\bf A}(y[1],\dots,y[t]) - {\bf A}(x[1],\dots,x[t])\|_{\ell_2} \leq C_t\eta^{t/2}\sqrt{k\min({\bf d})}.
\eeqn
Since there are at most $\prod_{s=1}^t{n\choose c_s}c_s^{c_s} \leq n^{2(c_1+\cdots+c_t)} \leq n^{2t\max({\bf d})/\eta}$ tuples $(y[1],\dots,y[t])$ as above, the result follows.
\end{proof}

\subsection{Putting things together}

\begin{proof}[ of Lemma~\ref{lem:smallprocess}]
Lemma~\ref{lem:maurey} shows that for any $\eps \geq 1$, we have
\beq\label{eq:net1}
N\big(\Lambda_{\bf d}, \ell_2, \eps\sqrt{k\min({\bf d})}\big) \leq n^{2\max({\bf d})/\eps^{2/t}}.
\eeq

In addition, for $\eps > 0$, the size of any $\eps\sqrt{kn}$-net is bounded above by the cardinality of~$\Lambda_{\bf d}$, and therefore 
\beqn\label{eq:net2}
N\big(\Lambda_{\bf d}, \ell_2, \eps\sqrt{k\min({\bf d})}\big) 
\leq
{n\choose d_1}\cdots{n\choose d_t} \leq n^{2t\max({\bf d})}.
\eeqn  

Denote $\Delta = \diam(\Lambda_{\bf d})$.
By Corollary~\ref{cor:dudley-lambda}, a substitution of variables, and Proposition~\ref{prop:diameter}, 
\begin{align*}
\int_0^{\Delta}\sqrt{\frac{\log N(\Lambda_{\bf d}, \ell_2, \eps)}{2tk\min({\bf d})\max({\bf d})\log n}}d\eps
&\leq
\int_0^1d\eps + 
\int_1^{ \Delta/\sqrt{k\min({\bf d})} } \eps^{-1/t}d\eps\\
&\leq
1 +
\frac{ \Delta^{1 - 1/t} }{ (1-1/t)(k\min({\bf d}))^{1/2 - 1/(2t)} }\\
&\leq
C\min({\bf d})^{1/2 - 1/(2t)}.
\end{align*}
Hence,
there is an absolute constant~$C\in (0,\infty)$
such that the left-hand side of~\eqref{eq:process-val}  is at most
\beqn
C\int_0^{\Delta}\sqrt{\log N(\Lambda_{\bf d}, \ell_2, \eps)}d\eps
\leq
C'\sqrt{k\max({\bf d})\log n}\min({\bf d})^{1 - 1/(2t)}.
\eeqn 
\end{proof}

\bibliographystyle{alphaabbrv}
\bibliography{arith_exp.bib}

\appendix

\section{Proof of Theorem~\ref{thm:matrix-Hoeffding}}
\label{sec:matrix-Hoeffding}

For completeness, we derive Theorem~\ref{thm:matrix-Hoeffding} (the matrix Hoeffding bound) from the following special case of a result of Tomczak-Jaegermann~\cite[Theorem~3.1]{Tomczak-Jaegermann:1974}.

\begin{theorem}[Tomczak-Jaegermann]\label{thm:TJ}
There exists an absolute constant~$C \in (0,\infty)$ such that the following holds.
Let~$k$ be a positive integer and let~$p \geq 2$.
Let~$A_1,\dots,A_k \in \R^{n\times n}$. 
Then,
\beqn
\Exp\Big[\Big\|\sum_{i=1}^k \epsilon_i A_i\Big\|_{S_\infty}\Big]
\leq 
C\sqrt{\log n}
\Big(\sum_{i=1}^k \|A_i\|_{S_\infty}^2\Big)^{1/2}.
\eeqn
\end{theorem}


\begin{proof}[ of Theorem~\ref{thm:matrix-Hoeffding}]
We proceed just as in the proof of Theorem~\ref{thm:tensor-hoeffding}.
For $C$ as in Theorem~\ref{thm:TJ}, let $\sigma = C\sqrt{8(\log n)/k}$ and $\alpha = e/(4\sigma^2)$.
Define
\beqn
Z = \Big\| \frac{1}{k}\sum_{i=1}^k (A_i - \Exp[A_i]) \Big\|_{S_\infty}.
\eeqn
By Markov's inequality,
\begin{align}
\Pr[Z > \eps]
=
\Pr\big[e^{\alpha Z^2} > e^{\alpha \eps^2}\big]
\leq
e^{-\alpha \eps^2}\Exp\big[e^{\alpha Z^2}\big].\label{eq:markov}
\end{align}

By Lemma~\ref{lem:symmetrization}, Theorem~\ref{thm:Kahane-Khintchine} and Theorem~\ref{thm:TJ}, for every integer~$p\geq 1$, 
$
\big(\Exp[Z^p]\big)^{1/p}
\leq
\sigma\sqrt{p}$.
The result then follows since
\beqn
\Exp\big[e^{\alpha Z^2}\big]
\leq
1 + \sum_{p=1}^{\infty}\Big(\frac{2 \alpha \sigma^2}{e}\Big)^p = 2.\qedhere
\eeqn
\end{proof}

\end{document}